\documentclass[11pt]{amsart}
\usepackage[mathscr]{eucal}
\usepackage{times}
\usepackage{amsmath,amssymb,latexsym,amscd}   
\usepackage{hyperref}
\hypersetup{colorlinks=true,linkcolor=red,citecolor=brown,linktocpage=true}
\usepackage[all,cmtip]{xy}
\usepackage{fancyhdr}
\usepackage{mathrsfs}
\usepackage{upgreek}
\usepackage{xcolor}
\usepackage{tikz}
\usepackage{graphicx}
\usepackage{tikz-cd}
\usepackage{enumerate}
\usepackage{qtree}
\usepackage{tikz-qtree}
\usepackage[style=numeric]{biblatex}
\DeclareMathOperator{\ler}{ler}

\DeclareMathOperator{\pt}{pt}

\DeclareMathOperator{\op}{op}
\theoremstyle{plain}
\newtheorem*{theorem*}{Theorem}
\newtheorem{thm}{Theorem}[section]

\newtheorem{lem}[thm]{Lemma}
\newtheorem{prop}[thm]{Proposition}

\theoremstyle{definition}
\newtheorem{dfn}[thm]{Definition}

\newtheorem{ej}[thm]{Example}
\newcommand\adj[2]{{
  \ar[from=#1, to=#2,"\dashv"marking,draw=none]
}}

\begin{document}
\title[Article Title]{On non-archimedean frames}

\author{M. Bocardo-gaspar$^1$}
\address[1,4]{Department of Mathematics \\University of Guadalajara \\ Blvd. Gral. Mar-celino Garc\'{i}a Barrag\'{a}n \#1421, Ol\'{i}mpica 44430 \\ Guadalajara, Jalisco, M\'{e}xico}
\email{miriam.bocardo@academicos.udg.mx}
\email{luis.zaldivar@academicos.udg.mx}

\author{J. Urenda $^2$}
\address[2]{Department of Mathematical Sciences \\ University of Texas at El Paso \\ 500 West University Avenue, 79968  \\ El Paso, Texas, USA}
\email{jcurenda@utep.edu}

\author{F. Ávila $^3$}
\address[3]{Department of Physics and Mathematics \\ Autonomous University of Ciudad Ju\'{a}rez \\ Av. Plutarco El\'{i}as Calles \#1210, 32310 \\Ciudad Ju\'{a}rez, Chihuahua, M\'{e}xico}
\email{favila@uacj.mx}

\author{A. Zaldívar $^4$}


\begin{abstract}
In this investigation, we introduce the class of non-archimedean frames in spirit with the topological notion of non-archimedean spaces. We attach a base to each non-archimedean frame that constitutes a tree, and then we show that every non-archimedean frame is a quotient of the frame of opens of the tree's branch space. Finally, we also explore the spatiality of a non-Archimedean frame and give partial answer to this problem.
\end{abstract}


\maketitle

\section{Introduction}\label{intro}

In the category of frames one can describe some frames by generators and relations. For instance, in \cite{avila2020frame} the author introduces frame of $p$-adic numbers $\mathcal{L}(\mathbb{Q}_p)$, which is the free frame generated by the elements $B_r(a)$, where $a\in \mathbb{Q}$ and $r\in|\mathbb{Q}|:=\{p^{-n}\mid n\in \mathbb{Z}\}$, subject to the following relations:
\begin{itemize}
	\item[(1)] $B_r(a) \wedge B_s(b)=0$ whenever $|a-b|_p\geq r\vee s$. 
	\item[(2)] $1=\bigvee\{B_r(a):a\in\mathbb{Q},r\in |\mathbb{Q}|\}$.
	\item[(3)] $B_r(a)=\bigvee\{B_s(b)\colon |a-b|_p< r, s<r, r\in |\mathbb{Q}|\}$.
\end{itemize}
Note that these relations imply that the set  \[\mathcal{B}=\{B_{r}(a)\colon r\in|\mathbb{Q}|, a\in\mathbb{Q}\}\] is a base for $\mathcal{L}(\mathbb{Q}_p)$. On the other hand, let $B_r(a)$, $B_s(b)$ be any two elements in $\mathcal{L}(\mathbb{Q}_p)$ and, without loss of generality, assume that $s\leq r$. Then, if $|a-b|_p\geq r$, we have  $B_r(a)\wedge B_s(b)=0$ by relation (1), and if $|a-b|_p< r$, we have $B_s(b)\leq B_{r}(a)$ by relation (3). Thus, for any $B_r(a), B_s(b) \in \mathcal{B}$, either
\[B_r(a)\wedge B_s(b)=0\, \text{ or }\, B_s(b)\leq B_{r}(a)\, \text{ or }\, B_s(b)\geq B_{r}(a). \]

Now, recall that a \emph{non-archimedean} topological space $S$ is a Hausdorff space with a base $\mathcal{B}$ satisfying the \emph{trichotomy} laws: if $B_1$, $B_2\in\mathcal{B}$, we have that either $B_1\cap B_2=\varnothing \text{ or } B_1\subseteq B_2 \text{ or }  B_2\subseteq B_1$ holds 
(see, e.g., \cite{monna1950remarques}, \cite{nyikos1975some}, \cite{nyikos1975structure}, \cite{nyikos1999some}). Motivated by this, we say that a frame is non-archimedean if it has a (non-archimedean) base $\mathcal{B}$ that satisfies these trichotomy laws: if  $b_{1},b_{2}\in\mathcal{B}$, then either
$b_{1}\wedge b_{2}=0$ or $b_{1}\leq b_{2}$ or $b_{2}\leq b_{1}$ holds.

It follows that  $\displaystyle{\mathcal{B}:=\{B_{r}(a):r\in|\mathbb{Q}|, a\in\mathbb{Q}\}}$ is a non-archimedean base for the frame $\mathcal{L}(\mathbb{Q}_p)$, and that $\mathcal{L}(\mathbb{Q}_p)$ is a non-archimedean frame.

The frame of the Cantor set, denoted by $\mathcal{L}(\mathbb{Z}_p)$, is another example of a non-archimedean frame. It is defined using the fact that the $p$-adic integers $\mathbb{Z}_p$ is homeomorphic to the Cantor set for every prime number $p$. In \cite{avila2022cantor}, it is shown that $\mathcal{L}(\mathbb{Z}_p)$ is a $0$-dimensional, (completely) regular, compact, and metrizable frame whose Cantor-Bendixson derivative is zero (see, e.g., \cite{Simmons2014Cantor}).

Moreover, for every non-archimedean normed field $(\mathbb{K},|\cdot|_\mathbb{K})$ we can define $\mathcal{L}(\mathbb{K})$ in a similar fashion. Then, $\mathcal{L}(\mathbb{K})$ is a non-archimedean frame. Moreover, if $S$ is a non-archimedean topological space, then $\mathcal{O}(S)$ is a (regular) non-archimedean frame. 

To the best of our knowledge, there is no literature on non-Archimedean frames motivating, in essence, the development of this work. In this manuscript, we introduce a general point-free theoretic framework (for non-archimedean structures) to address some classical topological situations in the broader context of frames.

The presentation of the manuscript is as follows: Section \ref{pre} contains the background needed to understand the rest of the manuscript.
Section \ref{non} defines the class of frames called \emph{non-archimedean}. In particular, we establish that every non-archimedean base has a tree structure  and prove one of the main results:

\begin{center}
\emph{A frame is non-archimedean if and only if it is a quotient of a topology}\\ \emph{of a branch space of a tree.} 

\end{center}

 Finally, the last section is devoted to introduce a type of monotone bar induction for trees , which allow us to obtain partial result on the spatiality of non-archimedean frames.

\section{Preliminaries}\label{pre}

In this section, we recall some basic facts from frame (locale) theory. For details, we refer to \cite{Johnstone82} and \cite{PicadoPultr}.

A \emph{frame} (locale, complete Heyting algebra) is a complete lattice \[(A,\leq, \bigvee,\wedge,0,1)\] in which the following distributive law holds:

\[a\wedge(\bigvee X)=\bigvee\{a\wedge x\mid x\in X\},\] for every $X\subseteq A$ and $a\in A$.
Frames are, in specific ways, algebraic manifestations of topological spaces. The ubiquitous example of a frame is the topology of any topological space.

We denote by $\mathrm{Frm}$ the category of frames, with frames as objects and arrows are functions $h\colon A\rightarrow B$ such that

\begin{equation*} \label{frmhom}
\, h(a\wedge b)=h(a)\wedge h(b),\text{ and } h\left( \bigvee_{i\in J} a_i\right)=\bigvee_{i\in J} h(a_i), \text{for any set } J.
\end{equation*} As the theory suggests, there is an intrinsic relationship between the categories of topological spaces and frames. More specifically, there exists an adjunction between these categories:

\[
  \begin{tikzcd}
    \mathrm{Top}\ar[d,shift right=2,"\mathcal{ O}"'{name=L}]
    \\
\mathrm{Frm}^{\op}\ar[u,shift right=2,"\pt"'{name=R}]
\adj{L}{R}
\end{tikzcd}
\]
where $\mathcal{O}(S)$ is the lattice of open sets of the space $S$ and $\pt(A)$ is the \emph{point space} of the frame, which can be described in three different ways: as frame characters or completely prime filters or $\wedge-$irreducible elements. For more details, see, e.g., \cite{PicadoPultr}.

Every frame $A$ has a Heyting implication $(\_\succ\_)$ given by \[x\leq(a\succ b)\Leftrightarrow x\wedge a\leq b.\] Thus, for every element $a\in A$, its \emph{negation} is given by \[\neg a=(a\succ 0).\] That is, $\neg a$ is the greatest element such that $a\wedge x=0$. An element $a\in A$ is \emph{complemented} if there exists an element (necessarily unique) $b\in A$ such that $a\vee b=1$ and $a\wedge b=0$; in fact, the complement of a complemented element $a$ is $\neg a$. 

As usual, the concept of base in topological spaces has its interpretation in frames: A subset $\mathcal{B}$ of a frame $A$ is a \emph{base} if, for each $a\in A$, we have \[a=\bigvee\{b\in\mathcal{B}\mid b\leq a\},\]
and we say that a frame is \emph{zero dimensional} if it has a base of complemented elements.

A central theme in pointfree topology is that fundamental topological notions, such as regularity, continuity, separation, can be generalized from topological spaces to frames (and locales), often in a more elegant and logically simpler form. In particular, the classical separation axioms have their counterparts in frames (see, e.g., \cite{picado2021separation}). We recall some of these notions that will be useful in this manuscript.

We say that $A$ is \emph{regular} if for any $a,b\in A$, with $a\nleq b$, there are $x,y\in A$ such that \[a\vee x=1,\;\;\; y\nleq b,\;\; x\wedge y=0.\]
Equivalently, we define $a\prec b$  iff $\neg a \vee b=1$ (the way below relation) and one say that a frame is regular if \[a=\bigvee\{x\in A\mid x\prec a\}
,\] for every $a\in A$.

This definition agrees with the classical one: A topological space $S$ is regular if and only if the frame $\mathcal{O}(S)$ is regular.

We say that $a$ is \emph{completely below} $b$, denoted $a\prec\!\!\prec b$,  if there are $a_r \in A$, with $r$ rational and $0\leq r \leq 1$, such that
\[a_0=a,\;\;\;\; a_1=b,\;\;\text{ and }\;\; a_r\prec a_s\, \text{ for } r<s.\]
Then, a frame $A$ is said to be \emph{completely regular} if
\[a=\bigvee\{x\mid x\prec\!\!\prec a,\}\]  for every $ a\in A$.
This is also an extension of the corresponding notion in classical topology: A topological space $S$ is completely regular if and only if the frame $\mathcal{O}(S)$ is completely regular.

Crucial to the algebraic study of frames is the notion of a \emph{nucleus}, that is, a function $j\colon A\rightarrow A$ such that:
\begin{enumerate}[(i)]
\item $a\leq b\Rightarrow j(a)\leq j(b)$.
\item $a\leq j(a)$ for all $a\in A$.
\item $j(a)\wedge j(b)\leq j(a\wedge b).$
\item $j^{2}=j$.
\end{enumerate}

The set of all nuclei, denoted by $\mathcal{N}(A)$, is a frame which is often called \emph{the assembly}. For every $j\in \mathcal{N}(A)$, the set of fixed elements $A_{j}=\{a\mid j(a)=a\}$ is a frame; the function $j^{*}\colon A\rightarrow A_{j}$, defined by $a\mapsto j(a)$, is a surjective frame morphism, and every quotient of $A$ is identified with the fixed set of a nucleus.

Another essential tool is a \emph{coverage}, which is a generalization of the notion of a topology that allows for systematic generation of frames by specifying relations of covering.  
This concept was introduced in \cite{Johnstone82} and further developed in \cite{simmons2004coverage} and \cite{gambino2007spatiality} (see \cite{ball2016dedekind} for some uses).
Although coverages are typically defined using downsets, the dual notion employing upsets is conceptually equivalent and can be used interchangeably. The upset approach is particularly well-suited for trees, as their natural hierarchical structure aligns intuitively with the concept of covering by upsets.

A \emph{coverage} on a poset $(\EuScript{T},\leq)$ is a relation $\vdash$ between the set of  upsets of $\EuScript{T}$, denoted by $\mathcal{U}(\EuScript{T})$, and elements of $\EuScript{T}$. It is well known that any coverage defines an inflator (a monotone, inflatory function) $d\colon\mathcal{U}(\EuScript{T})\rightarrow\mathcal{U}(\EuScript{T})$ given by:
\[a\in d(U)\Leftrightarrow U\vdash a,\] for every $U\in\mathcal{U}(\EuScript{T})$ and $a\in A$. 
If the coverage satisfies specific rules (see after definition \ref{cov}), we obtain a nucleus and, thus, a quotient of $\mathcal{U}(\EuScript{T})$. This is explored in \cite{simmons2004coverage} and \cite{simmons2007coverage} (we use these ideas after Theorem \ref{branchi} to untangle the spatiality of non-archimedean frames).

Finally, free frames and quotient frames define structures by generators and relations. We take the quotient of the free frame $\mathbb{F}$ on the set $\mathcal{B}$ of generators modulo the relation generated by the pairs $(u,v)$ for the stipulated relations $u=v$ (or $u\leq v$). When taking this quotient, one describes the relations by a system of formal equalities, i.e., if $u, v$ are formulas to identify, we write $u=v$ instead of $uRv$. 

Note that defining morphisms from a presented frame  $A/R \rightarrow B$ it is enough to have a morphism $h\colon A \to B$ for which $uRv$ implies $h(u)=h(v)$  and then a factorization define a morphism $A/R \rightarrow B$ (see, e.g. \cite{banaschewski1997real}, \cite{PicadoPultr}, \cite{picado2021separation}).

\section{Non-archimedean frames}\label{non}

Motivated by the classic topological notion of non-archimedean space (as previously introduced) we give the frame-theoretic analogue.
\begin{dfn}\label{fr}
Let $A$ be a frame, a base $\mathcal{B}$ of $A$ is called \emph{non-archimedean} if  $b_{1},b_{2}\in\mathcal{B}$, then either
\begin{itemize}
\item[(N1)]
$b_{1}\wedge b_{2}=0$, or
\item[(N2)]
$b_{1}\leq b_{2}$, or
\item[(N3)]
$b_{2}\leq b_{1}$.
\end{itemize}
holds.
\end{dfn}
\begin{dfn}\label{non1}
A \emph{non-archimedean frame} is a regular frame with a non-archimedean base.
\end{dfn}

As usual, we will see what happens with the quotients of a non-archimedean frame:

\begin{prop}\label{quot}

Let $A$ be an non-archimedean frame then $A_{j}$ is a non-archimedean frame for every $j\in \mathcal{N}(A)$.

\end{prop}

This straightforward statement gives a lot of examples. For instance, as we mentioned above, if $S$ is a non-archimedean space, then its topology $\mathcal{O}S$ is a non-archimedean frame, and then for every nucleus $j\in \mathcal{N}(\mathcal{O}S)$ we have that $\mathcal{O}S_{j}$ is a non-archimedean frame.

\begin{lem}
Let $A$ be a non-archimedean frame with base $\mathcal{B}$ and let $a\in \mathcal{B}$. If there is $b\neq 0$ such that $b\prec a$, then $a$ is complemented.
\end{lem}

\begin{proof}
Suppose there is $b\neq 0$ with $b\prec a$. Since $\mathcal{B}$ is a base, then \[\neg b =\bigvee\{x\in \mathcal{B}\mid x\wedge b=0\}.\]
Since $A$ is non-archimedean, for each $x\in\mathcal{B}$, we have that either $x\leq a$ or $a\leq x$ or $x\wedge a=0$. Note that if $a\leq x$ and $x\wedge b=0$, then $b\leq a \leq x\leq \neg b$ which is impossible. Thus, 
\[\neg b=\bigvee\{x\in\mathcal{B}\mid x\wedge b=0,\, x\leq a\}\vee\bigvee\{x\in\mathcal{B}\mid x\wedge b=0,\, x\wedge a=0\}.\]
Also note that
\[\bigvee\{x\in\mathcal{B}\mid x\wedge b=0,\, x\leq a\}\leq a \text{ and }  \bigvee\{x\in\mathcal{B}\mid x\wedge b=0,\, x\wedge a=0\}\leq \neg a.\]
So,
\begin{align*}
\neg b&=\bigvee\{x\in\mathcal{B}\mid x\wedge b=0,\, x\leq a\}\vee\bigvee\{x\in\mathcal{B}\mid x\wedge b=0,\, x\wedge a=0\}\\
&\leq \neg a \vee a.
\end{align*}
Since $1=\neg b\vee a\leq \neg a\vee a$, it follows that $a$ is complemented, as required.

\end{proof}

For example, the classical notion of non-archimedean space implies that the base consists of complemented elements, the following result is the point free analogue of this property (see \cite{nyikos1975structure}, \cite{nyikos1999some} and it seems that the first appearance of these spaces was in  \cite{monna1950remarques}).

\begin{thm}\label{zero}
Every non-archimedean frame is zero-dimensional; in particular, every non-archimedean frame is completely regular.
\end{thm}

Our first goal is to prove the point-free counterpart of  \cite[Theorem 2.7]{nyikos1999some} to do that we require some preliminaries, we start with the next proposition (compare with \cite[Theorem 2.3]{nyikos1999some} and \cite{davis1978spaces}).

\begin{prop}\label{noinfty}
Let $A$ be a non-archimedean frame with base $\mathcal{B}$ and let $\mathcal{C}$ be a subset of a chain in $\mathcal{B}$, then $\bigwedge \mathcal{C}$ is complemented.
\end{prop}
\begin{proof}

If $\bigwedge \mathcal{C}=0$, $\bigwedge \mathcal{C}$ is complemented. Assume $\bigwedge \mathcal{C}\neq 0$ and let $b\in \mathcal{B}$ such that $b\leq \bigwedge \mathcal{C}$. Note that if $b'\wedge b =0$, then we have two possibilities:
\begin{itemize}
\item[(1)] $b'\leq c$ for all $c\in \mathcal{C}$, which implies $b'\leq \bigwedge \mathcal{C}$.
\item[(2)] $b'\nleq c$ for some $c\in \mathcal{C}$. That is, there is $c\in \mathcal{C}$ such that $b'\wedge c =0$ or $c\leq b'$, but the latter case would imply that $b\wedge b'=b\neq 0$ (since $b\leq c \leq b'$), thus $b'\wedge c=0$ for some $c\in \mathcal{C}$. 
\end{itemize}
Then,
\begin{align*}
\bigwedge \mathcal{C} &=  \bigwedge \mathcal{C} \wedge (b\vee \neg b) = \big(\bigwedge \mathcal{C} \wedge b\big) \vee   \big(\bigwedge \mathcal{C} \wedge \neg b\big)\\
&=b\vee \left(\bigvee \left\{ \bigwedge \mathcal{C}\wedge b' \mid b'\in \mathcal{B} \text{ and } b'\wedge b=0 \right\} \right)\\
&=b\vee \left(\bigvee \left\{ \bigwedge \mathcal{C}\wedge b' \mid b'\in \mathcal{B},\, b'\wedge b=0,\, b'\leq \bigwedge \mathcal{C} \right\} \right)\\
& \hspace{0.4in}\vee \left(\bigvee \left\{ \bigwedge \mathcal{C}\wedge b' \mid b'\in \mathcal{B},\, b'\wedge b=0,\, b'\wedge c = 0 \text{ for some } c\in \mathcal{C} \right\} \right)\\
&=b\vee \left(\bigvee \left\{b'\in \mathcal{B} \mid b'\wedge b=0,\, b'\leq \bigwedge \mathcal{C} \right\} \right)\\
\end{align*}
Now,
\begin{align*}
1&=b\vee \left(\bigvee \left\{b'\in \mathcal{B} \mid b'\wedge b=0,\, b'\leq \bigwedge \mathcal{C} \right\} \right)\\
& \hspace{0.5in} \vee  \left(\bigvee \left\{b'\in \mathcal{B} \mid b'\wedge b=0 \text{ and } b'\wedge c =0 \text { for some } c\in \mathcal{C} \right\} \right)\\
&\leq \bigwedge \mathcal{C} \vee \left(\neg\bigwedge \mathcal{C} \right).
\end{align*}
Therefore, $\bigwedge \mathcal{C}$ is complemented.
\end{proof}






It is worth noticing that, given an element \(b\) and a chain
\(C\) in a non-Archimedean base, we have
\[
b\wedge\bigvee C=
\begin{cases}
\quad 0            &\iff b\wedge c = 0 \quad (\forall\,c\in C),\\[4pt]
\quad b            &\iff b\le c        \quad (\text{for some }c\in C),\\[4pt]
\displaystyle\bigvee C &\iff c\le b        \quad (\forall\,c\in C).
\end{cases}
\]

This leads to the following result.

\begin{prop}\label{base}
Every non-Archimedean base $\mathcal{B}$ extends to a non-Archimedean base
$\mathcal{B}$, on the same frame, that is closed under suprema of
chains.
\end{prop}

\begin{proof}
Suppose \(b_{1}=\bigvee C\) and \(b_{2}=\bigvee D\) for chains
\(C,D\subseteq \mathcal{B}\).  
If \(b_{1}\wedge b_{2}\neq 0\) then either \(b_{1}\le d\) for some
\(d\in D\) (which implies \(b_{1}\leq b_{2}\)) or, for every \(d\in D\), there exists \(c\in C\) such that
\(d\le c\); in the latter case \(b_{2}\le b_{1}\).

Define
\[
\overline{\mathcal{B}}\;=\;\bigl\{\bigvee C \mid C\subseteq \mathcal{B}\text{ is a
  chain}\bigr\}
\]
and let \(D\) be a chain in \(\overline{\mathcal{B}}\).

\medskip
\noindent\emph{Case 1: \(D\) has a maximum element \(m\).}  
Then \(m\) is the supremum of a chain in \(B\), hence
\(\bigvee D = m \in \overline{\mathcal{B}}\).

\medskip
\noindent\emph{Case 2: \(D\) has no maximum element.}  
Let \(d\in D\), then there is \(d^{+}\in D\) with \(d<d^{+}\). Write \(d^{+}=\bigvee C_{d}\) for some chain \(C_{d}\subseteq \mathcal{B}\).
Because \(d\le\bigvee C_{d}\), there exists
\(c_{d}\in C_{d}\) with \(d\le c_{d}\).
Thus, we obtain
\[
\bigvee_{d\in D} c_{d}
   \;\le\;
\bigvee D
   \;=\;
\bigvee_{d\in D} d
   \;\le\;
\bigvee_{d\in D} c_{d},
\]
so \(\bigvee D = \bigvee_{d\in D} c_{d} \in \overline{\mathcal{B}}\).

Thus \(\overline{\mathcal{B}}\) is closed under suprema of chains and, by the first
paragraph,  \(\overline{\mathcal{B}}\) is a non-Archimedean base.
\end{proof}

\begin{lem}[Canonical form]\label{nonar1}
Let $(A,\mathcal{B})$ be a non-archimedean frame with base closed under suprema of chains, then every $a\in A$ is a join of disjoint members of $\mathcal{B}$.

\end{lem}

\begin{proof}
For every $a\in A$ let $\mathcal{B}_{a}=\{b\in\mathcal{B}\mid b\leq a\}$ since by the hypothesis every upper bound of any chain of $\mathcal{B}_{a}$ is in $\mathcal{B}_{a}$, then we can apply Zorn´s lemma to see that every element of $\mathcal{B}_{a}$ lies below a maximal member and these are disjoint.
\end{proof}
 	

With this at hand, we can see that every non-archimedean frame has a tree associated, but first, we need some definitions.

\begin{dfn}\label{tree2}

A \emph{tree} $\EuScript{T}$ consist  of the following:
\begin{itemize}
\item[(1)] $(\EuScript{T},\leq)$, a partial ordered set (we call $a\in\EuScript{T}$ a node).

\item[(2)] For each node $a$, the set $\{x\in\EuScript{T}\mid a>x\}$ of immediate predecessors of $a$ is well-ordered.

A maximal chain of the tree $\EuScript{T}$ is called a \emph{brach} of the tree.
\end{itemize}

\end{dfn}

We will see that every non-archimedean frame has a base winch and a tree with the reverse order of $A$.
\begin{thm}\label{tree}
Let $A$ be a non-archimedean frame with base $\mathcal{B}$. Then $A$ has a base, which is a tree.
\end{thm}

\begin{proof}

We define $\EuScript{T}(\mathcal{B})$ by transfinite recursion. Let $\EuScript{T}(0)$ be any nontrivial decomposition of the top element in $A$ into disjoint members of $\mathcal{B}$ (using Lemma \ref{nonar1}). 
For a successor ordinal $\alpha+1$, suppose $\EuScript{T}(\alpha)$ is defined. For each non-atomic element $b\in\EuScript{T}(\alpha)$, produce non-trivial decomposition $W(b)$ of $b$ and set 
\[\EuScript{T}(\alpha + 1) = \bigcup_{\stackrel{b \in\EuScript{T}(\alpha)}{b \text{ non-atomic}}} W(b).\]

Recall that $\EuScript{T}(\alpha)$ contains only disjoint members, then for each different
$n, n' \in\EuScript{T}(\alpha+1)$, either  there is $p \in\EuScript{T}(\alpha)$, with $n,n' < p$, or there are different, hence disjoint, 
$p,p' \in \EuScript{T}(\alpha)$, with $n < p$ and $n' < p'$; in any case, $n \wedge n' = 0$.

Therefore $\EuScript{T}(\alpha+1)$ consist only of disjoint members of $\mathcal{B}$.

For the limit ordinal, that is, \[\lambda=\bigvee\{\alpha\mid\alpha<\lambda\}.\]

Consider $\bigwedge\EuScript{C}$ where $\EuScript{C}$ is a chain given by \[c\in\EuScript{C}\Leftrightarrow \text{ there exists }\alpha<\lambda\text{ such that }c\in\EuScript{T}(\alpha).\]

Let $\EuScript{T}(\lambda)$ be the collection of nontrivial infima of these chains.

If $\EuScript{T}(\alpha)$ is pair-wise
disjoint for each $\alpha < \lambda$, then suppose that $\EuScript{C}$ and $\EuScript{C}'$ are different chains such that $\bigwedge\EuScript{C}$ and $\bigwedge\EuScript{C}'$ are in $\EuScript{T}(\lambda)$. Then, there is an ordinal
$\alpha_0 < \lambda$ such that $c\in\EuScript{C}$ and $c'\in\EuScript{C}'$ such that, $c,c'\in\EuScript{T}(\alpha_{0})$ such that $c\neq c'$. By the induction hypothesis, $c \wedge c'=0$ thus
\[\bigwedge\EuScript{C} \wedge \bigwedge\EuScript{C}'=0.\]
		
By cardinality, the above process stops at an ordinal $\gamma$, where $\EuScript{T}(\gamma)= \emptyset$. Notice that $ \EuScript{T}(\mathcal{B})=\bigcup\{\EuScript{T}(\beta)\mid\beta<\gamma\}$
and from the construction, we have that for each node the set of predecessors is well-ordered, $\EuScript{T}(\mathcal{B})$ is a tree.

To show $\EuScript{T}(\mathcal{B})$ is a base for $A$, let $b \in \mathcal{B}$ and consider the least ordinal $\delta = \delta(b)$ such that there is no $n \in \EuScript{T}(\delta)$ with $b < n$.
Such an ordinal exists; otherwise, $\EuScript{T}(\gamma)\neq \varnothing$. If $\delta = \alpha+1$, then there is a unique $n \in\EuScript{T}(\alpha)$ with 
$b < n$, by the decomposition $n = \bigvee W(n)$, with $W(n) \subset \EuScript{T}(\delta)$. Note that for each $w \in W(n)$, either $w \le b$ or 
		$w \wedge b = 0$, then 
		\begin{align*}
			b &= b \wedge \bigvee W(n) \\ 
			  &= \bigvee_{w \in W(n)} b \wedge w \\
			  &= \bigvee_{\stackrel{w \in W(n)}{w \le b}} w.
		\end{align*}
In case $\delta$ is a limit ordinal, there is a unique $n \in\EuScript{T}(\delta)$ such that $b \not < n$ since the members of $\EuScript{T}(\delta)$ are 
		disjoint and for the defining chain $\{c_\alpha\}_{\alpha < \delta}$, with $\bigwedge c_\alpha = n$, we have $b < c_\alpha$ for all 
		$\alpha < \delta$. If $n$ is atomic, then $n \in \mathcal{B}$ and $b=n$; otherwise, $n$ admits a decomposition $n = \bigvee W(n)$ with
		$W(n) \subset \mathcal{B}$ as in the former case. This shows that $\EuScript{T}(\mathcal{B})$ is a base for $A$ since $a=\bigvee\{\bigvee W(b)\mid b\leq a \text{ and } b\in\EuScript{T}(\alpha)\}.$

\end{proof}

The construction of free frames is described in \cite[Appendix, Section 9]{picado2021separation}, we will follow the same method. Given a (rooted) tree \(\EuScript{T}\) under reverse inclusion, its \emph{tree frame} \(\mathcal{L}(\EuScript{T})\)  is the free frame generated by its nodes subject to the following  relations:
\begin{itemize}
\item[(1)] \(b=\bigvee\{b'\in \EuScript{T} \mid b'\leq b \}.\)
\item[(2)] \(b_1\wedge b_2 = 0\,\) if and only if \(\,b_1\nleq b_2 \text{ and } b_2\nleq b_1. \) 
\item[(3)] \(\top=\bigvee\{b \mid b \in \EuScript{T} \}.\)
\end{itemize}
Observe that, in the frame $\mathcal{L}(\EuScript{T})$, one has \[b\wedge b'\neq 0\, \text{ if and only if } \,b\leq b'\text{ or } b'\leq b.\]
Moreover, if every non terminal node (a leaf) in $\EuScript{T}$ has at least two descendants, then $\mathcal{L}(\EuScript{T})$ is zero dimensional since each basic element is complemented -- such proof would follow by a standard transfinite induction argument on the height of the tree -- and thus  $\mathcal{L}(\EuScript{T})$ is  a non-archimedean frame.

\begin{dfn}\label{freetree}
Let $(A,\mathcal{B})$ be a non-archimedean frame and let $\EuScript{T}(\mathcal{B})$ be a base for $A$ which is a tree (as in Theorem \ref{tree}). The \emph{free frame}  $\mathcal{L}(\EuScript{T}(\mathcal{B}))$ is  the tree frame generated by the tree base  \(\EuScript{T}(\mathcal{B})\).
\end{dfn}




Now we will see that there exists an intrinsic relation between the frame $\mathcal{L}(\mathcal{B})$ and the tree structure of a non-archimedean base of a frame with some topologies associated with the tree.

Consider any tree $(\EuScript{T},\leq)$ and let $\mathcal{U}(\EuScript{T})$ be the Alexandroff topology of the tree. Let $\Xi(\EuScript{T})$ be the set of branches of the tree and consider the map

\[k^{*}\colon\mathcal{U}(\EuScript{T})\rightarrow\mathcal{P}(\Xi(\EuScript{T})),\] given by \[\xi\in k^{*}(U)\Leftrightarrow\xi\cap U\neq\emptyset,\] with $\xi\in\Xi(\EuScript{T})$. Note that the set $k^{*}[\mathcal{U}(\EuScript{T})]$ is a topology on $\Xi(\EuScript{T})$, denoted by $\mathcal{O}(\Xi(\EuScript{T}))$, and the function \[k^{*}\colon\mathcal{U}(\EuScript{T})\rightarrow\mathcal{O}(\Xi(\EuScript{T}))\] is a surjective frame morphism. Thus, it has a right adjoint \[k_{*}\colon\mathcal{O}(\Xi(\EuScript{T}))\rightarrow\mathcal{U}(\EuScript{T}),\] given by $k_{*}(V)=\bigcap\{\EuScript{T}-\xi\mid\xi\in\Xi(\EuScript{T})-V\}$. As usual, \[k_{*}k^{*}\colon\mathcal{U}(\EuScript{T})\rightarrow\mathcal{U}(\EuScript{T})\] is a nucleus, denoted by ${\rm ker}=k_{*}k^{*}$. Then, \[\mathcal{U}(\EuScript{T})_{{\rm ker}}\cong\mathcal{O}(\Xi(\EuScript{T})).\]

Now, for each $t\in\EuScript{T}$, let $U_{t}=\{\xi\in\Xi(\EuScript{T})\mid t\in\xi\}$. These sets, indexed by the points of the tree, constitute a base for a topology 
for $\Xi(\EuScript{T})$. In fact, observe that if $\xi\in U_{t}$, then $\uparrow \! t\cap\xi\neq\emptyset$, therefore $\xi\in k^{*}(\uparrow \! t)$. That is, $U_{t}\subseteq k^{*}(\uparrow \! t)$. Moreover, if $\xi\in\bigcup\{U_{x}\mid x\in\uparrow \! t\}$, then $\xi\in U_{x}$ for some $x\in\uparrow \! t$. Therefore, $\xi\in k^{*}(\uparrow \! t)$ and thus $\bigcup\{U_{x}\mid x\in\uparrow \! t\} \subseteq k^{*}(\uparrow \! t)$.

Conversely, if $\xi\in k^{*}(\uparrow \! t)$, then $\xi\in\bigcup\{U_{x}\mid x\in\uparrow \! t\}$. Thus, \[k^{*}(\uparrow \! t)=\bigcup\{U_{x}\mid x\in\uparrow \! t\}.\]

Since every upset $U=\bigcup\{\uparrow \! u\mid u\in U\}$ and $k^{*}$ is a frame morphism, we have that every element in $\mathcal{O}(\Xi(\EuScript{T}))$ is a join of $U_{t}$ for each $t\in U$. Therefore, we have the following:

\begin{lem}\label{base1}
With the notation as above, we have that the topology generated by $\{U
_{t}\mid t\in\EuScript{T}\}$ is the topology given by $k^{*}$. That is, \[\langle U
_{t}\mid t\in\EuScript{T}\rangle=\mathcal{O}(\Xi(\EuScript{T})).\]

\end{lem}

Indeed, with this topology, we have a non-archimedean space, 
many examples of non-archimedean spaces are \emph{branch spaces of a tree} (well-founded trees). As the author shows in \cite[Theorem 2.10]{nyikos1999some}, every non-archimedean space is a subspace of a branch space. In particular, every branch space of a tree is a non-archimedean space. We have a similar situation in the point-free setting.
At this point, we turn back to the case of a non-archimedean frame $(A,\mathcal{B})$, with the observations above. We have the following gadgets:

\begin{itemize}
\item[(T1)] The tree base $\EuScript{T}(\mathcal{B})$ of Theorem \ref{tree}.
\item[(T2)] The topology $\mathcal{O}(\Xi(\EuScript{T}(\mathcal{B})))$ of the set of branches.
\item[(T3)] The free frame $\mathcal{L}(\EuScript{T}(\mathcal{B}))$.
\end{itemize}
For the sake of notation, we will use $\EuScript{T}$ instead of  $\EuScript{T}(\mathcal{B})$ (keep in mind we started with a base \(\mathcal{B}\) of a non-archimedean frame $A$).

An essential (and perhaps obvious) remark is the following:

\begin{lem}\label{basereverse}
With the above notation, let $\mathcal{B}$ be a non-archimedian base of $A$. For each $x,y\in\mathcal{B}$, let $U_{x}$ and
 $U_{y}$ be their corresponding basis sets in $\mathcal{O}\Xi(\EuScript{T})$. Then, \[x\leq y\Leftrightarrow U_{y}\subseteq U_{x}.\]

\end{lem}

\begin{proof}
Let $x$ and $y$ basic elements. Then, $x\leq y$ or $y\leq x$ (the case $x\wedge y=0$ is not included). In any case, consider 
\[\upharpoonright \!(x)=\uparrow \!(x)-\{x\}\] and \[\downharpoonright \!(y)=\downarrow \!(y)-\{y\}.\] These two sets are, respectively, a filter and an ideal that meets in the empty set. Therefore, an application of Zorn's lemma gives, in any case, a maximal branch that gives the equivalence with $U_{y}\subseteq U_{x}$ (or $U_{x}\subseteq U_{y}$
).
\end{proof}


\begin{thm}\label{spat}
Let $\EuScript{T}$ be a (rooted) tree. Then, the frame $\mathcal{L}(\EuScript{T})$ is spatial. In fact, \[\mathcal{L}(\EuScript{T})\cong\mathcal{O}(\Xi\EuScript{T}),\] 
and the isomorphism $\mathfrak{U}\colon\mathcal{L}(\EuScript{T})\rightarrow
\mathcal{O}(\Xi\EuScript{T})$ is given by \[\mathfrak{U}(b)=\bigcup\{U_{x}\mid x\leq b\}, \text{ for each } b\in \EuScript{T}.\]
\end{thm}

\begin{proof}
Let us see first that $\mathfrak{U}\colon\mathcal{L}(\EuScript{T})\rightarrow
\mathcal{O}(\Xi\EuScript{T})$ is a frame morphism. To verify this, we see that this function sends the relations T1 - T3 into identities. Since $\mathfrak{U}$ is monotone and $\mathfrak{U}(a\wedge b)=\mathfrak{U}(a)\cap\mathfrak{U}(b)$, for $a,b\in \EuScript{T}$, it follows that $\mathfrak{U}$ respects $a\leq b$, $b\leq a$, and $a\wedge b=\bot$.

Now, let $b\in \EuScript{T}$ such that $b\leq \top=\bigvee\EuScript{T}$, then $\mathfrak{U}(\top)=\Xi(\EuScript{T})$. Finally, $b=\bigvee\{b'\leq b\mid b'\in\EuScript{T}\}$, then $\mathfrak{U}(b)	=\bigcup\{U_{x}\mid x\leq b\}$. Therefore, $U_{b'}\leq\mathfrak{U}(b)$. On the other hand, consider any branch $\xi\in U_{x}$ for some $x$ with $x\leq\bigvee\{b'\leq b\mid b'\in\EuScript{T}\}$. Then, there exists $b'$ such that $x\leq b'$, and thus $\xi\in\bigcup\{U_{x}\mid x\leq b'\}$. Therefore, $\mathfrak{U}$ send the relations into identities and so it is a frame morphism.

Now we show that $\mathfrak{U}$ is surjective. Let $V\in\mathcal{O}(\Xi\EuScript{T})$ be any open set, then there exists an upper section $U\in\mathcal{U}(\EuScript{T})$ such that $k^{*}(U)=V$. Let \[\EuScript{T}(U)=\{x\in\EuScript{T}\mid x\in U\}.\] Since $\mathfrak{U}(\bigvee\EuScript{T}(U))=\bigcup\{U_{x}\mid x\leq\bigvee\EuScript{T}(U)\}
,$ then we have $U_{x}\leq k^{*}(U)$ by the definition of $k^{*}$. Then, $\bigcup\{U_{x}\mid x\leq\bigvee\EuScript{T}(U)\}\subseteq k^{*}(U)$ and thus, it follows that $\mathfrak{U}(\bigvee\EuScript{T}(U))\subseteq k^{*}(U)$. On the other hand, if $\xi\in k^{*}(U)$, then there exists $x\in\xi\cap U$. Therefore, $x\leq\bigvee\EuScript{T}(U)$ and thus, $\xi\in U_{x}$, that is, $\mathfrak{U}(b)=V$, where $b=\bigvee\EuScript{T}(U)$.

Finally, we show $\mathfrak{U}$ is injective; in fact, we can see that there exist enough branches that \emph{separate points}, i.e., if $b<b'$, then there exists a branch $\xi$ such that $b'\in\xi$ and $x\notin\xi$ for every $x\leq b$. To see this, let  $b\neq b'$. First, if $b=0$ and $b'\neq 0$, then $\mathfrak{U}(b)=\emptyset$ and $\mathfrak{U}(b')\neq \emptyset$. Next, if $b\wedge b'=0$, then $\mathfrak{U}(b)\wedge\mathfrak{U}(b')=\emptyset$. Thus, $\mathfrak{U}(b)\neq \mathfrak{U}(b')$. Now, we consider the case $b<b'$ (the other case is similar). Consider the interval $(b,b']$. Note that this is an interval of the non-archimedean base, thus for any elements $x,y\in (b,b'] $, we have $x\leq y$ or $y\leq x$, and the case $x\wedge y=0$ can not occur since $b\neq 0$ and $b< x\wedge y$. Thus, $(b,b']$ is a chain. Set \[\mathcal{A}=\{\text {chains }\EuScript{C}\text{ of }\EuScript{T}\text{ such that } b'\in\EuScript{C}\text{ and } x\notin\EuScript{C}\; \;\forall x<b\}\]
and observe that $\mathcal{A}$ is not empty ($(b,b'] \in\mathcal{A}$). Let $\mathcal{C}\subseteq\mathcal{A}$ be directed and consider $\bigcup\mathcal{C}$. This element is a chain for if $x,y\in\bigcup\mathcal{C}$, then there are chains of $\mathcal{C}$, say $\EuScript{C}$ and $\EuScript{C}'$, with $x\in\EuScript{C}$ and $y\in\EuScript{C}'$. Since $\mathcal{C}$ is directed, we can find a chain $\EuScript{D}$ such that $\EuScript{C},\EuScript{C}'\subseteq\EuScript{D}$, so $x,y\in\EuScript{D}$. Since $\EuScript{D}$ is a chain, we have $x\leq y$ or $y\leq x$ and thus $\bigcup\mathcal{C}$ is a chain. Clearly, this chain is an element of $\mathcal{A}$, thus it satisfy the hypotheses of Zorn's lemma. This ensure the existence of a maximal chain $\xi\in\mathcal{A}$, that is, a branch of our tree such that $b'\in\xi$ and $x\notin\xi$ for every $x\leq b$. It follows that $\mathfrak{U}(b)\subsetneq\mathfrak{U}(b')$, as required.
\end{proof}

As an immediate consequence of the above, we have the following fact (which can be thought of as the point-free analog of Theorem 2.10 in \cite{nyikos1999some}):

\begin{thm}\label{quotarch}
A frame $A$ is non-archimedean if and only if $A$ is a quotient of a topology of a branch space of a tree. 
\end{thm}

\begin{proof}
If $A$ is non-archimedean then, by Theorem \ref{tree} and Theorem \ref{spat}, there exists a tree $\EuScript{T}$ with associated branch space $\Xi(\EuScript{T})$ and $\mathcal{O}(\Xi\EuScript{T})\cong\mathcal{L}(\EuScript{T})$. Then, for any open $V$ in $\mathcal{O}\Xi(\EuScript{T})$, we have $V=k^{*}(U)$, where $U$ an upset of the tree. Thus, we can define $\eta\colon\mathcal{O}(\Xi\EuScript{T})\rightarrow A$ by $\eta(V)=\bigvee\{x\in\EuScript{T}\mid U_{x}\subseteq V\}$.

Observe that this function is surjective, i.e., if $a\in A$, then there exists a set $\mathfrak{B}_{a}$ of elements of $\EuScript{T}$ such that $b\leq a$ and $\bigvee\mathfrak{B}_{a}=a$. Then, for each $b\in\mathfrak{B}_{a}$ we can consider the corresponding $U_{b}$ in $\mathcal{O}\Xi(\EuScript{T})$. Let \[V=\bigcup\{U_{b}\mid b\in\mathfrak{B}_{a}\},\] then if $U=\{ b\mid U_{b}\subseteq V \}$ (which is an upset) we have $k^{*}(U)=V$ and thus \[\eta(V)=a.\] 
Since $\eta(U_b)=b$, for each $b$ in the basis of $\mathcal{O}(\Xi(\EuScript{T}))$, therefore $\eta$ is a frame morphism.
For the converse, observe that if we start with a tree $(\EuScript{T},\sqsubseteq)$, then the associated branch space $\mathcal{O}\Xi\EuScript{T}$, with the topology given by $k^{*}[\mathcal{U}(\EuScript{T})]$, is non-archimedean. Thus, the frame $\mathcal{O}\Xi(\EuScript{T})$ is non-archimedean in the sense of Definition \ref{non}. Then, by Proposition \ref{quot}, the result follows.
\end{proof}

\section{Coverages on trees of non-archimedian frames  }\label{gbarind}

It is natural to ask when a non-archimedean frame is spatial, or to ask under what conditions it happens that $A\cong\mathcal{L}(\EuScript{T})$.
At first, we erroneously believed that every non-archimedean frame was spatial, which (fortunately) is a wrong statement, as the following example shows:

Consider the Cantor frame $\mathcal{L}(\mathbb{Z}_{2})$ (see \cite{avila2022cantor}). This frame is non-archimedean and, by Proposition \ref{quot}, every quotient of  $\mathcal{L}(\mathbb{Z}_{2})$ is non-archimedean; in particular, the quotient $\mathcal{L}(\mathbb{Z}_{2})_{\neg\neg}$ (the regular opens of the Cantor) is a non-archimedean frame. If every non-archimedean frame is spatial, this would imply that $\mathcal{L}(\mathbb{Z}_{2})_{\neg\neg}$ is spatial, which is not true (in fact, $\pt(\mathcal{L}(\mathbb{Z}_{2})_{\neg\neg})=\emptyset$). In this chapter, we develop a partial answer to this question, starting with the following preparatory material.\\

Consider the morphism $\eta\colon\mathcal{O}\Xi(\EuScript{T})\rightarrow A$ given in Theorem \ref{quotarch}. Then, it has a right adjoint \[\eta_{*}\colon A\rightarrow\mathcal{O}\Xi(\EuScript{T}).\] 
Let us denote by ${\rm ler}=\eta_{*}\eta\colon\mathcal{O}\Xi(\EuScript{T})\rightarrow\mathcal{O}\Xi(\EuScript{T})$ the corresponding nucleus, that is, \[\mathcal{O}\Xi(\EuScript{T})_{{\rm ler}}\cong A.\]
By adjointness, we know that
\[\eta_{*}\colon A\rightarrow \mathcal{O}\Xi(\EuScript{T}),\] where\[\eta_{*}(a)=\bigcup\{U\in\mathcal{O}\Xi(\EuScript{T})\mid \eta(U)\leq a\}.\]

Now, consider the function $\beta\colon A\rightarrow\mathcal{O}(\Xi(\EuScript{T})),$ given by \[\beta(a)=\bigcup\{U_{x}\mid x\leq a\}.\] Note that $\beta=\eta_{*}$.
Therefore, \[{\rm ler}(V)=\bigcup\left\lbrace U_{b}\mid b\leq\bigvee\{x\mid U_{x}\subseteq V\}\right\rbrace.\]

Since $\mathcal{O}(\Xi(\EuScript{T}))$ is a quotient of $\mathcal{U}(\EuScript{T})$, the frame $A$ is a quotient of the latter and the corresponding nucleus is given by $k_{*}\eta_{*}\eta k^{*}\colon\mathcal{U}(\EuScript{T})\rightarrow\mathcal{U}(\EuScript{T})$. To have a friendlier description of this operator, in this section we introduce a coverage, that is, a relation between elements of $\mathcal{U}(\EuScript{T})$ and  $\EuScript{T}$. A coverage relation is a powerful tool since there is one-to-one correspondence between coverages on a partially ordered set $\EuScript{T}$ and (pre)nuclei on $\mathcal{U}(\EuScript{T})$, for more details, see, e.g., \cite{Johnstone82}, \cite{simmons2007coverage} and \cite{simmons2004coverage}. 

Recall that a \emph{coverage} \(\,\vdash\,\) on a partially ordered set $(B,\leq)$ is a relation between the set of upsets $\mathcal{U}(B)$ of $B$ and elements of $B$ satisfying the following rules:\\
For each $a,b\in B$ and $U,V\in\mathcal{U}(B)$,
\begin{itemize}

\item[(R1)]
$U\vdash a,\;\;\; a\leq b\Rightarrow U\vdash b$.

\item[(R2)] $a\in U\Rightarrow U\vdash a.$ 
\item[(R3)] 

$a\leq b,\;\;V\vdash b,\;\; V\subseteq U\Rightarrow U\vdash a$.
\item[(R4)]$U\vdash a,\;\;\; V\vdash a\Rightarrow U\cap V\vdash a$. 
\item[(R5)] $V\vdash a,\;\;\; V\vdash U\Rightarrow U\vdash a$, 
where $V\vdash U$ means $(\forall u\in U)[V\vdash u]$.
\end{itemize}

There is a bijective correspondence between the coverage relations on $(B,\leq)$ and nuclei on the frame $\mathcal{U}(B)$, this is given by:
\[U\vdash a\Longleftrightarrow a\in j(U).\]

Recall that ${\rm ker}=k_{*}k^{*}\colon\mathcal{U}(\EuScript{T})\rightarrow\mathcal{U}(\EuScript{T})$ is the nucleus associated to $\mathcal{O}\Xi(\EuScript{T}),$ where the corresponding coverage relation $\vdash$ is given by \[U\vdash a\Longleftrightarrow (\forall\xi\in\Xi(\EuScript{T}))[a\in\xi\Rightarrow\xi\cap U\neq\emptyset]\Longleftrightarrow  a\in {\rm ker}(U).\]
The following diagram summarizes the quotients with the respective morphisms:
\[
\begin{tikzcd}
\mathcal{U}(\EuScript{T})
 \ar[loop left, "\ker=k_{*}k^{*}"] \ar[d, "\ker^{*}"']\ar[loop right, "\ler=k_{*}\eta_{*}\eta k^{*}"] \ar[d, "\ker^{*}"']
\\
\mathcal{O}\Xi(\EuScript{T})\ar[d, "(k_{*}\eta_{*}\eta k^{*})^{*}
"']\\ A\end{tikzcd}
\]

We introduce the following coverage relation.

\begin{dfn}\label{cov}
With the above notation, set

\[U\Vdash a\Longleftrightarrow (\forall\xi\in\Xi(\EuScript{T}))[\exists x\leq a][x\in\xi\Rightarrow\xi\cap U\neq\emptyset].\]

\end{dfn}

Now we see that the above relation is a coverage and  defines  ${\rm ler}$.

\begin{lem}\label{cov1}
With the above notation, the relation \[U\Vdash a\Longleftrightarrow (\forall\xi)[\exists x\leq a][x\in\xi\Rightarrow\xi\cap U\neq\emptyset]\] is a coverage and its corresponding nucleus is ${\rm ler}=k_{*}\eta_{*}\eta k^{*}$. 
\end{lem}

\begin{proof}
Let $U\in\mathcal{U}(\EuScript{T})$, then 

\begin{align*}
k_{*}\eta_{*}\eta k^{*}(U)=\bigcap\big\{\EuScript{T}-\xi\mid \xi\in\Xi(\EuScript{T})-\eta_{*}\eta k^{*}(U) \big\}&=\\
\bigcap\Big\{\EuScript{T}-\xi\mid \xi\in\Xi(\EuScript{T})-\bigcup\big\{U_{x}\mid x&\leq\bigvee\{x'\mid U_{x'}\subseteq k^{*}(U)\}\big\} \Big\}.
\end{align*}
Thus, for an element $x\in\EuScript{T}$, we have \[x\notin\bigcap\left\{\EuScript{T}-\xi\mid \xi\in\Xi(\EuScript{T})-\bigcup\big\{U_{x}\mid x\leq\bigvee\{x'\mid U_{x'}\subseteq k^{*}(U)\}\big\}\right\}\] if and only if there exists a branch $\xi$ such that $x\in\xi$ and \[\downarrow \! \! \big(\bigvee\{y\mid U_{y}\leq k^{*}(U)\}\big)\nsubseteq\xi,\] which is the negation of \[U\Vdash a\Longleftrightarrow (\forall\xi)[\exists x\leq a][x\in\xi\Rightarrow\xi\cap U\neq\emptyset]\Longleftrightarrow  a\in {\rm ler}(U).\]

It is straightforward to verify that $\Vdash$ is coverage, i.e., it satisfies the rules right after Definition \ref{cov}. Therefore, \[{\rm ler}=k_{*}\eta_{*}\eta k^{*}\text{ as required. }\]
\end{proof}

We notice that \[{\rm ker}\leq {\rm ler}.\]

Thus, we have $\mathcal{U}(\EuScript{T})_{{\rm ler}}\cong A$. Now, we want to determine when these two quotients of $\mathcal{U}(\EuScript{T})$, $\mathcal{O}\Xi(\EuScript{T})$ and $A$, are equal. To show this, we need a preamble.

The CBD-process of the tree $\EuScript{T}$ is the Cantor-Bendixson derivative of the space $(\EuScript{T},\mathcal{U}(\EuScript{T}))$. In this case, this derivative can be describe as follows: Given any tree $(\EuScript{T},\leq)$, for each node $a\in\EuScript{T}$, there is a unique ordinal number $\alpha(a)$ such that the lower section $\downarrow \! \!(a)$ has at most $\alpha(a)+1$ members.  This information gives a stratification of the tree into levels (we use $\#$ for cardinality):\[\mathbb{L}_{\alpha}=\{x\mid \alpha(a)+1\geq \#\downarrow \! \!(x)\}.\]

In other words, the stratification is constructed through heights. Therefore, our tree has a bound given by the least ordinal below $\#\EuScript{T}$. 

Now, for each node $a$, let \[\upharpoonright \!(a)=\{x\mid x\in\uparrow \!(a)\text{ and } \alpha(x)=\alpha(a)+1\}.\] 
Consider any $U\in\mathcal{U}(\EuScript{T})$ and let ${\rm der}(U)$ be the upset given by \[a\in {\rm der(U)}\Longleftrightarrow\upharpoonright \!(a)\subseteq U.\]
Then, 
\begin{align*}
a\notin \rm{der}(U) 
&\Leftrightarrow\upharpoonright \!(a)\nsubseteq U \\
&\Leftrightarrow  (\exists x)[a<x\in U']\\
&\Leftrightarrow a\in U'\text{ and } a \text{ is not maximal in } U'\\
&\Leftrightarrow a\in \lim(U'),
\end{align*}
where $\lim(U')$ is the set of limit points of $U'$ (the complement of $U$). That is, $\rm{der}(U)=\lim(U')'$. Thus, $\rm{der}$ is the Cantor-Bendixson derivative of $\mathcal{U}(\EuScript{T})$. This is a prenucleus, and its ordinal closure derivative $\rm{der}^{\infty}$ is a nucleus. The ordinal $\infty$ is the Cantor-Bendixson rank of the tree. Also, observe that the coverage corresponding to $\rm{der}$ is 

\[U\vDash a\Longleftrightarrow (\forall x)[a\leq x\Rightarrow x\in U]\Longleftrightarrow  a\in \rm{der}(U).\]
Since 
\[U\vdash a\Longleftrightarrow (\forall\xi)[a\in\xi\Rightarrow\xi\cap U\neq\emptyset]\Longleftrightarrow  a\in {\rm ker}(U),\]
we have that $\rm{der}\leq {\rm ker}$, and thus $\rm{der}\leq {\rm ker}\leq {\rm ler}$. Now, we say that $U\in\mathcal{U}(\EuScript{T})$ is \emph{perfect} if $\rm{der}(U)=U$, which is equivalent to \[U\vDash a\Longleftrightarrow a\in U\] or \[\upharpoonright \!(a)\subseteq U\Rightarrow a\in U.\] This condition can be thought as a \emph{generalized} Bar induction. For more details, see, e.g., \cite{gambino2007spatiality}. 

\begin{dfn}\label{genbar}
For a tree base $\EuScript{T}$ of a non-archimedean frame $A$, $\EuScript{T}$ satisfies the \emph{generalized monotone bar induction}, denoted as $\mathrm{GBI}(\EuScript{T})$ (or simply $\mathrm{GBI}$), if \[U\Vdash a \Rightarrow a\in U\] for each perfect $U\in\mathcal{U}(\EuScript{T})$ and $a\in\EuScript{T}$.
\end{dfn}
Obviously, $\mathrm{GBI}(\EuScript{T})$ is equivalent to \[{\rm ler}(U)=U\Longleftrightarrow \rm{der}(U)=U.\] Note that $U$ is perfect if and only if $U\in\mathcal{U}(\EuScript{T})_{\rm{der}}=\mathcal{U}(\EuScript{T})_{\rm{der}^{\infty}}$. Since \[\rm{der}^{\infty}\leq {\rm ker}\leq {\rm ler},\] then, if $U$ is perfect and $\mathrm{GBI}(\EuScript{T})$ holds, we have \[\rm{der}^{\infty}={\rm ker}={\rm ler},\] which in quotient language says \[\mathcal{U}(\EuScript{T})_{\rm{der}^{\infty}}\cong\mathcal{O}(\Xi(\EuScript{T}))\cong A.\]
The following theorem summarizes these comments.  Consequently, it gives some conditions on the injectivity of the morphism $\eta\colon\mathcal{O}(\Xi(\EuScript{T}))\rightarrow A$. In general, we do not have a characterization for non-archimedean frames $A$ for which $\eta$ is injective. Nevertheless, we have the following interesting fact.

\begin{thm}\label{branchi}
Let $(A,\EuScript{T}(\mathcal{B}))$ be any non-archimedean frame with tree base $\EuScript{T}(\mathcal{B})$. Then, the following conditions are equivalent:

\begin{enumerate}[(i)]
\item $\mathrm{GBI}(\EuScript{T}(\mathcal{B}))$ holds.
\item ${\rm der}^{\infty}={\rm ker}={\rm ler}$.
\item If ${\rm der}(U)=U$, then ${\rm ler}(U)=U$ with $U\in\mathcal{U}(\EuScript{T})$.
\item $\mathcal{U}(\EuScript{T})_{{\rm der}}$ is a spatial frame.
\end{enumerate}
\end{thm}

\begin{proof}
Suppose (i) and consider any $U$ such that ${\rm der}^{\infty}(U)=U$. Since $\mathrm{GBI}$ holds, we have $U\Vdash a \Rightarrow a\in U$. That is, ${\rm ler}(U)=U$. The other comparison follows immediately from ${\rm der}\leq {\rm ker}\leq {\rm ler}$.

Note that, since $\mathcal{U}(\EuScript{T})_{{\rm der}^{\infty}}=\mathcal{U}(\EuScript{T})_{{\rm der}}$, we have that (ii) if and only if (iii).

If (iii) holds, we have \[\mathcal{U}(\EuScript{T})_{{\rm der}^{\infty}}\cong\mathcal{O}(\Xi(\EuScript{T}))\cong A.\] In particular, this amounts to the spatiality of $\mathcal{U}(\EuScript{T})_{{\rm der}}$, which is exactly (iv).

Finally, we show that (iv) implies $\mathrm{GBI}$. If (iv) holds, we have that $\mathcal{U}(\EuScript{T}(\mathcal{B}))_{{\rm der}^{\infty}}$ has enough points separating distinct elements. These points are points of $\mathcal{U}(\EuScript{T}(\mathcal{B}))$ for which ${\rm der}(\EuScript{P})=\EuScript{P}$. Then, for each of these points, we have $\lim(\EuScript{P}')'=\EuScript{P}$, that is, $\EuScript{P}'$ is a directed lower section with no maximal elements. Now, consider $U\Vdash a$ and suppose $a\notin U$. Then, $\downarrow \! \!(a)\nsubseteq U$. Since $\downarrow \! \!(a)\subseteq\uparrow \!\downarrow \! \! (a)$ and ${\rm der}$ inflates,  we conclude that \[{\rm der}(\uparrow \!\downarrow \! \!(a))\nsubseteq U.\]

Thus, the spatial condition ensures the existence of a (separating) point  $\EuScript{P}$ of $\mathcal{U}(\EuScript{T}(\mathcal{B}))_{{\rm der}}$ such that ${\rm der}(\uparrow \!\downarrow \! \!(a))\nsubseteq\EuScript{P}$ and $U\subseteq\EuScript{P}$. Since ${\rm der}(\uparrow \!\downarrow \! \!(a))\nsubseteq\EuScript{P}$, there exists $x\in {\rm der}(\uparrow \!\downarrow \! \!(a))$ such that $x\notin\EuScript{P}$. By definition of ${\rm der}$, we have \[x\in {\rm der}(\uparrow \!\downarrow \! \!(a))\Leftrightarrow \,\uparrow \!(x)\subseteq\uparrow \!\downarrow \! \!(a).\] Hence, there exists $e\leq a$ such that $e\leq x$ ($x\in\uparrow \!\downarrow \! \!(a)$). Therefore, $e\in\EuScript{P}'$ since $x\in\EuScript{P}'$. If we show that there exists a branch $\xi\subseteq\EuScript{P}'$ with $e\in\xi$, then we get a contradiction.

Since $\EuScript{P}'$ is a lower section with no maximal elements, then $e\in\EuScript{P}'$ is non-maximal. Thus, there exists $s_{1}\in\EuScript{P}'$ such that $e<s_{1}$. In this way, we construct a chain in $\EuScript{P}'$ \[\EuScript{C}(e)=e<s_{1}<\ldots<s_{i}<\ldots\]

Now, consider the set $\mathcal{C}_{e}=\{\EuScript{C}\subseteq\EuScript{P}'\mid\EuScript{C}\text{ is a chain with } e\in\EuScript{C}\}$. This partially ordered set is not empty due to $\EuScript{C}(e)\in\mathcal{C}_{e}$. It is clear that  the union of any directed family $\mathfrak{C}\subseteq\mathcal{C}_{e}$ is an upper bound in $\mathcal{C}_{e}$. Thus, by Zorn's Lemma, we have a maximal chain $\xi$ in $\mathcal{C}_{e}$. This implies  that  $e\in\xi$ and $\xi\cap U=\emptyset$, which is a contradiction.
\end{proof}

It is important to highlight that we do not have yet a general criteria to determine whenever a non-archimedean frame is spatial or not.


The following examples give some classical non-archimedean frames.
\begin{ej}\label{inductive tree}
Some well-known trees: 

\begin{enumerate}[(i)]

\item Cantor tree $\EuScript{C}$:
\[
\begin{tikzpicture}[grow'=up][sibling distance=72pt]
\Tree [.$\perp$ [.$\bullet$ 
[.$\bullet$    ]   [.$\bullet$  ] ]
[.$\bullet$ [.$\bullet$ $\ldots$ ]
[.$\bullet$  ] ] ]
\end{tikzpicture}
\]

\item Baire tree  $\EuScript{B}$:

\[\begin{tikzpicture}[grow'=up]
  \Tree [.$\perp$
    [.$\bullet$ [.$\bullet$ $\ldots$ ]
    [.$\bullet$ ] [.$\bullet$ ] [.$\bullet$ ] [.$\ldots$ ] ]
    [.$\bullet$ ] [.$\bullet$ ] [.$\bullet$ ] [.$\ldots$ ]
  ]
\end{tikzpicture}\]
  
\item Köing Tree  $\EuScript{K}$:

\[\begin{tikzpicture}[grow'=up]
  \Tree [.$\perp$
    [.$\bullet$  ]
    [.$\bullet$ 
    [.$\bullet$ ] [.$\bullet$  ] [.$\bullet$ [.$\bullet$ [.$\bullet$ $\ldots$ ] [.$\bullet$ ] [.$\bullet$ ] [.$\ldots$ ] ] ] ] [.$\bullet$ ] [.$\bullet$ [.$\ldots$ ] ] 
  ]
\end{tikzpicture}\]

\end{enumerate}

\end{ej}

In these classic examples, it is well-known the following:
\begin{enumerate}[(i)]
\item  The Cantor-Bendixson rank of $\EuScript{C}$ is at most $\omega$.

\item The Cantor-Bendixson rank of $\EuScript{B}$ is $\Omega$ (the least uncountable ordinal).

\item The Cantor-Bendixson rank of $\EuScript{K}$ is $\omega+1$.  
\end{enumerate}

If $\EuScript{T}\in\{\EuScript{C},\EuScript{B},\EuScript{K}\}$,  the non-archimedean frame $A_{\EuScript{T}}=\mathcal{L}(\EuScript{T})$ satisfies  $\pt\mathcal{L}(\EuScript{T})\cong\Xi(\EuScript{T})$ and  $A_{\EuScript{T}}\cong\mathcal{O}\Xi(\EuScript{T})$. Basically, this is the content in the constructive setting of  \cite{gambino2007spatiality}) or, in general, it is a consequence of our theory in the presence of full choice.

It is important to mention that $\mathcal{L}(\EuScript{C}
)\cong\mathcal{O}(\mathbb{Z}_{p})\cong\mathcal{L}(\mathbb{Z}_{p})$,  where $\mathcal{O}(\mathbb{Z}_{p})$ is the ultrametric topology, for more details \cite{avila2022cantor}. In \cite{avila2020frame} and \cite{Avila}, the author introduces the frame of the $p$-adic numbers $\mathcal{L}(\mathbb{Q}_{p})$ and it is observed that this frame is spatial; in particular, it is isomorphic to the frame of opens of the topology of the non-archimedean space $\mathbb{Q}_{p}$. From that paper and our theory, we obtain the same result with the construction of a suitable tree for the base of that non-archimedean space. Therefore,  the topology of this tree is just $\mathcal{O}(\mathbb{Q}_{p})$, and thus \[\mathcal{L}(\mathbb{Q}_{p})\cong\mathcal{O}(\mathbb{Q}_{p}).\] 

\section{Acknowledgments}

We would like to express our sincere gratitude to Guram Bezhanishvili for his careful reading of this work. His insightful comments and constructive suggestions have significantly contributed to enhancing the quality of this manuscript. Research supported by CONAHCYT-PROJECT CBF2023-2024-2630.

\printbibliography
\end{document}